\documentclass{amsart}
\usepackage{relsize}
\usepackage{tcolorbox}
\usepackage{graphicx}
\usepackage{amssymb}
\usepackage[pagewise]{lineno}
\newtheorem{theorem}{Theorem}[section]
\newtheorem{lemma}[theorem]{Lemma}
\theoremstyle{definition}
\newtheorem{definition}[theorem]{Definition}
\newtheorem{corollary}{Corollary}[theorem]

\newcommand\ddfrac[2]{\frac{\displaystyle #1}{\displaystyle #2}}
\theoremstyle{remark}

\numberwithin{equation}{section}
\newcommand{\fallingfactorial}[1]{%
  ^{\underline{#1}}%
}

\title{The Generalized Superfactorial, Hyperfactorial and Primorial Functions} 
\author{Vignesh Raman}

\begin{document}
\maketitle
\begin{abstract}
This paper introduces a new generalized superfactorial function (referable to as $n^{th}$- degree superfactorial: $sf^{(n)}(x)${}) and a generalized hyperfactorial function (referable to as $n^{th}$- degree hyperfactorial: $H^{(n)}(x)${}), and we show that these functions possess explicit formulae involving figurate numbers. Besides discussing additional number patterns, we also introduce a generalized primorial function and 2 related theorems. Note that the superfactorial definition offered by Sloane and Plouffe (1995) is the definition considered (and not Clifford Pickover's (1995) superfactorial function: $n\$$).
\end{abstract}
\maketitle
\specialsection{\textbf{Introduction}}
\bigskip

Factorials, their numerous extensions into the complex region \footnote{Hadamard's Gamma function: $\Gamma(z)$ and the Luschny Factorial function: $L(z)$} and related functions \footnote{ Superfactorials: $sf(n)$; Hyperfactorials: $H(n)$; Multifactorials: $n!^{(k)}$} find themselves numerous applications in Combinatorics, Number Theory (NT), Functional Analysis, Probability Theory - the list goes on. Unfortunately, not much research has been done on factorial-related functions in recent literature due to two possible reasons:       
\vspace{3mm}
\begin{enumerate}
\item Current research in NT is more oriented towards Analytical NT. 
\vspace{1mm}
\item Current requirements for research and advancement in Combinatorics, NT and Diophantine Analysis (including algebraic manipulations) are satisfied by the current knowledge possessed about factorials and its closely related functions; it may be so that effort of the mathematical community is being solely directed towards research about functions behaving as extensions of the factorial function into other number systems (eg: \textit{Luschny factorial}), owing to their ready utility in applied math/real world applications.     
\end{enumerate}
\vspace{2mm}
Research pertaining to r-simplex (figurate) numbers and factorial-related functions motivated the investigation into their possible relation to  extensions of the superfactorial and hyperfactorial functions (formulated in this very paper).
\vspace{5mm}

We review certain factorial-related functions and figurate numbers in section 2, followed by the introduction of the generalized superfactorial function and a proof for its explicit formula in section 3. Section 4 will feature a similar treatment of the generalized hyperfactorial function. 

Section 5 will visually elucidate the aforementioned functions and explore two other number product-patterns involving factorials, and Section 6 will focus on an extension of the primorial function. Section 7 involves application of group theory and modular arithmetic lemmas to formulate two theorems concerning the superfactorial functions. Finally, Section 8 introduces an extension of the Legendre formula to shed light on the factorization of generalized superfactorials.
\specialsection{\textbf{Factorial-related functions and Figurate Numbers}}
\bigskip
\subsection{The Superfactorial Function}

\begin{definition}
The superfactorial function \textsuperscript{\cite{supfac}} is defined as the product of the first \textit{n} factorials, i.e.\begin{equation}
sf(n)=\prod^n_{k=1}k! \hspace{10mm}n \in \mathbb{Z}^+,\hspace{2mm}sf(0)=1
\end{equation}
It can also be defined recursively as follows:
\begin{equation*}
sf(n+1) = (n+1)!(sf(n));\hspace{3mm}n+1 \in \mathbb{Z}^+,\hspace{3mm} sf(0)=1
\end{equation*}
\end{definition}
\underline{Supplementaries}: 
\begin{itemize}
\item It is analogous to the tetrahedral numbers,- it represents the products of factorials akin to the tetrahedral numbers representing the sum of summations of successive natural numbers. 
\item It can also alternatively be represented as (on refactorization): \begin{equation*}
sf(n)=\prod^n_{k=1}k^{n-k+1},\hspace{3mm}n \in \mathbb{Z}^+
\end{equation*}
\item Its asymptotic growth is approximately\textsuperscript{\cite{sloane1}}: $$e^{\zeta'(-1)-\frac{3}{4}-\frac{3}{4n^2}-\frac{3}{2n}}*(2\pi)^{\frac{1}{2}+\frac{1}{2n}}(n+1)^{\frac{1}{2n^2}+n+\frac{5}{12}}$$
\end{itemize}
\bigskip
\subsection{The Hyperfactorial function}
\begin{definition}
The hyperfactorial function\textsuperscript{\cite{hypfac}} is defined as the product of numbers from $1$ to $n$ raised to the power of themselves, i.e \begin{equation}
H(n)=\prod^{n}_{k=1}k^{k},\hspace{3mm}n \in \mathbb{Z}^+
\end{equation}
It can also be defined as follows:
\begin{equation*}
H(n)=\prod^{n}_{k=1} {}^nP_k,\hspace{3mm}n \in \mathbb{Z}^+,\hspace{3mm}{}^nP_k=\frac{n!}{(n-k)!}
\end{equation*}
\end{definition}
\underline{Supplementaries}:
\begin{itemize}
\item The hyperfactorial sequence for natural numbers gives the discriminants for the probabilists' Hermite polynomial\textsuperscript{\cite{sloane2}}: $$H_{e_n}(x)=\frac{n!}{2\pi i}\oint_C\frac{e^{tx-\frac{tx^2}{2}}}{t^{n+1}}$$
\item Its asymptotic growth rate is approximately\textsuperscript{\cite{sloane3}}: \begin{equation*}
An^{(6n^2+6n+1)/12}e^{-n^2/4};\hspace{4mm}A{}\rightarrow \text{Glaisher Kinkelin constant}
\end{equation*}
\end{itemize}
\subsection{Figurate numbers}
\begin{definition}
Figurate numbers are those numbers that can be represented by a regular geometrical arrangement of equally spaced points.
\begin{figure}
\centering
\includegraphics{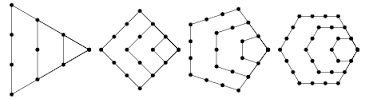}
\caption{Polygonal numbers}
\label{Fig:1}
\end{figure}
\end{definition}
\underline{Supplementaries}: \begin{itemize}
\item The $n^{th}$ regular $r$-polytopic number (i.e. figurate numbers for the r-dimensional analogs of triangles) is given by\textsuperscript{\cite{fignum}}:
\begin{equation}
P_r(n)=\binom{n+r-1}{r}
\end{equation}
\item The above formula essentially implies that (can be verified on addition of each term's respective binomial-coefficient representation): \begin{equation}
P_{r+1}(k)=\sum^k_{i=1}P_r(i) \label{eq:main}
\end{equation}
\item If the arrangement forms a regular polygon, the number is called a polygonal number. (Figure \ref{Fig:1})\textsuperscript{\cite{sloane4}}. 
\end{itemize}

\begin{center}
\line(1,0){360}
\end{center}
\specialsection{\textbf{The Generalized Superfactorial Function - GSF}}
\begin{definition}
Define the $n^{th}$-degree superfactorial as follows: 
\begin{equation}
sf^{(n)}(x)=\prod^{x}_{k=1}sf^{(n-1)}(k), \hspace{2mm}n \in \mathbb{Z}^+,\hspace{2mm}sf^{(0)}(x)\equiv x!
\end{equation}
We see that the first degree superfactorial is equivalent to the superfactorial definition given by Sloane and Plouffe (1995):
\begin{equation*}
sf^{(1)}(x)=sf(x)=\prod^{x}_{k=1}k!,\hspace{2mm}x \in \mathbb{Z}^+
\end{equation*}
\end{definition}
We now give an explicit formula for the GSF.
\begin{theorem}
The generalized superfactorial function is related to the r-simplex numbers (i.e.figurate numbers, $P_n(k) $) as follows:
\begin{equation*}
sf^{(n)}(x)=\prod^x_{k=1}[(x-k)+1]^{P_n(k)}=\prod^x_{k=1}k^{P_n((x-k)+1)};\hspace{2mm}x \in \mathbb{Z}^+; n \in \mathbb{Z}^{\geq} 
\end{equation*}
\end{theorem}
\begin{proof}
Theorem 3.2. is proven by using the Principle of Mathematical Induction.
\vspace{5mm}

\textbf{Base Case: $n=0$} \begin{equation*}
\begin{split}
sf^{(0)}(x) & = x!= \prod^x_{k=1}k^1 \hspace{5mm} [\text{As noted in section 2.1}]\\
& = \prod^x_{k=1}k^{P_0((x-k)+1)} \hspace{5mm} [\text{From equation (2.3), }P_0(n)=1] \\
& = 1^12^1\cdot \cdot \cdot(x-1)^1x^1= x^1(x-1)^1 \cdot \cdot \cdot 2^11^1 \\
& = x^{P_0(1)}(x-1)^{P_0(2)}(x-2)^{P_0(3)} \cdot \cdot \cdot 2^{P_0(x-1)}1^{P_0(x)} \\
& = \prod^x_{k=1}[(x-k)+1]^{P_0(k)}
\end{split}
\end{equation*}
\vspace{2mm}

\textbf{Inductive Hypothesis: }
\begin{equation*}
sf^{(m)}(x)=\prod^x_{k=1}[(x-k)+1]^{P_m(k)}
\end{equation*}
\vspace{1mm}

\textbf{To prove:} \begin{equation*}
sf^{(m+1)}(x)=\prod^x_{k=1}[(x-k)+1]^{P_{m+1}(k)}=\prod^x_{k=1}k^{P_{m+1}((x-k)+1)}
\end{equation*}
Now, 
\begin{equation*}
\begin{split}
sf^{(m+1)}(x) & =\prod^x_{k=1}sf^{(m)}(k) \hspace{5mm} [\text{From equation (3.1)}] \\
& = \prod^x_{k=1}\left[\prod^k_{i=1}[(k-i)+1]^{P_m(i)}\right] \hspace{5mm} [\text{From Inductive Hypothesis}] \\
& = \prod^1_{i=1}[(1-i)+1]^{P_m(i)}\prod^2_{i=1}[(2-i)+1]^{P_m(i)} \cdot \cdot \cdot \prod^x_{i=1}[(x-i)+1]^{P_m(i)} \\
&  = \left[1^{P_m(1)}\right]\left[2^{P_m(1)}1^{P_m(2)}\right]\left[3^{P_m(1)}2^{P_m(2)}1^{P_m(3)}\right] \cdot\cdot\cdot \left[x^{P_m(1)}(x-1)^{P_m(2)}(x-2)^{P_m(3)}\cdot\cdot\cdot 1^{P_m(x)}\right] \\
& = 1^{\left[\mathlarger{\mathlarger{\sum^x_{k=1}P_m(k)}}\right]}2^{\left[\mathlarger{\mathlarger{\sum^{x-1}_{k=1}P_m(k)}}\right]} \cdot \cdot \cdot x^{\left[\mathlarger{\mathlarger{\sum^1_{k=1}P_m(k)}}\right]} \hspace{5mm} [\text{On rearranging terms}] \\ 
& = \left[1^{P_{m+1}(x)}\right]\left[2^{P_{m+1}(x-1)}\right] \cdot \cdot \cdot x^{P_{m+1}(1)} \hspace{5mm} [\text{From equation \ref{eq:main}}] \\ 
& = \prod^x_{k=1}k^{P_{m+1}((x-k)+1)} = \prod^x_{k=1}[(x-k)+1]^{P_{m+1}(k)} 
\end{split}
\end{equation*}
\end{proof}
\specialsection{\textbf{The Generalized Hyperfactorial Function - GHF}}
\begin{definition}
Define the $n^{th}$-degree hyperfactorial as follows: \begin{equation}
H^{(n)}(x)=\prod^x_{k=1}H^{(n-1)}(k) \hspace{2mm}; \hspace{2mm} H^{(1)}(x)\equiv H(x)=\prod^x_{k=1}k^k
\end{equation}
\end{definition}
We now give an explicit formula for the GHF.
\begin{theorem}
The generalized hyperfactorial function is related to the r-simplex numbers as follows:
\begin{equation*}
H^{(n)}(x)=\prod^x_{k=1}k^{k[P_{n-1}((x-k)+1)]}; \hspace{2mm} n \in \mathbb{Z}^+
\end{equation*}
\end{theorem}
\begin{proof}
Theorem 4.2. is proven by using the Principle of Mathematical Induction.
\vspace{5mm}

\textbf{Base Case: $n=1$} \begin{equation*}
\begin{split}
H^{(1)}(x) & = \prod^x_{k=1}k^k = \prod^x_{k=1}k^{k[P_{0}((x-k)+1)]} \hspace{5mm} [\text{Since}\hspace{2mm} P_0(\xi)=1, \forall\hspace{1mm} \xi \in N ] \\
\end{split}
\end{equation*}
\vspace{5mm}

\textbf{Inductive Hypothesis: }
\begin{equation*}
H^{(m)}(x)=\prod^x_{k=1}k^{k[P_{m-1}((x-k)+1)]}
\end{equation*}
\vspace{1mm}

\textbf{To prove:} \begin{equation*}
H^{(m+1)}(x)=\prod^x_{k=1}k^{k[P_{m}((x-k)+1)]}
\end{equation*}
\begin{equation*}
\begin{split}
H^{(m+1)}(x) & =\prod^x_{k=1}H^{(m)}(k) \hspace{5mm} [\text{From equation (4.1)}] \\
& = \prod^x_{k=1}\left[\prod^k_{i=1}i^{i[P_{m-1}((k-i)+1)]} \right] \hspace{5mm} [\text{From Inductive Hypothesis}] \\
& = \prod^1_{i=1}i^{i[P_{m-1}((1-i)+1)]}\prod^2_{i=1}i^{i[P_{m-1}((2-i)+1)]} \cdot \cdot \cdot \prod^x_{i=1}i^{i[P_{m-1}((x-i)+1)]}   \\
& = \left[1^{1 \cdot P_{m-1}(1)}1^{1 \cdot P_{m-1}(2)} \cdot \cdot \cdot 1^{1 \cdot P_{m-1}(x)}\right] \left[2^{2 \cdot P_{m-1}(1)}2^{2 \cdot P_{m-1}(2)} \cdot \cdot \cdot 2^{2 \cdot P_{m-1}(x-1)}\right] \\
& \hspace{5mm} \cdot \cdot \cdot \left[ x^{x \cdot P_{m-1}(1)}\right] \hspace{5mm} [\text{On rearranging terms}] \\
& = 1^{1\cdot\left[\mathlarger{\mathlarger{\sum^x_{k=1}P_{m-1}(k)}}\right]}2^{2\cdot\left[\mathlarger{\mathlarger{\sum^{x-1}_{k=1}P_{m-1}(k)}}\right]} \cdot \cdot \cdot x^{x\cdot\left[\mathlarger{\mathlarger{\sum^1_{k=1}P_{m-1}(k)}}\right]} \\
& = 1^{1 \cdot P_m(x)}2^{2 \cdot P_m(x-1)} \cdot \cdot \cdot x^{x \cdot P_m(1)} \hspace{5mm} [\text{From equation \ref{eq:main}}] \\ 
& = \prod^x_{k=1}k^{k[P_m((x-k)+1)]}
\end{split}
\end{equation*}
\end{proof}
\specialsection{\textbf{Visual Exploration into Number Patterns}}
Intuition behind the formulae of the superfactorial and hyperfactorial can be developed on observation of the number pattern they manifest themselves into (note: \textit{product of all numbers in a pattern is the output of the function}). The below list will include some illustrations which will attempt to display the nature of the functions :
\begin{enumerate}
\item \begin{equation*}
\begin{split}
sf^{1}(5) = &1*2*3*4*5 \\
&1*2*3*4 \\
&1*2*3 \\
&1*2 \\
&1 \\
\end{split}
\end{equation*}
\item \begin{equation*}
\begin{split}
H^{1}(5) = &5*4*3*2*1 \\
&5*4*3*2 \\
&5*4*3 \\
&5*4 \\
&5 \\
\end{split}
\end{equation*}
\item \begin{equation*}
\begin{split}
sf^{2}(5) = &1*2*3*4*5 \text{\hspace{3mm}} 1*2*3*4 \text{\hspace{3mm}} 1*2*3 \text{\hspace{3mm}} 1*2 \text{\hspace{3mm}}1 \\
&1*2*3*4 \text{\hspace{8mm}} 1*2*3 \text{\hspace{8mm}} 1*2 \text{\hspace{8mm}} 1 \\
&1*2*3 \text{\hspace{13mm}} 1*2 \text{\hspace{13mm}} 1 \\
&1*2 \text{\hspace{18mm}} 1 \\
&1 \\
\end{split}
\end{equation*}
\item \begin{equation*}
\begin{split}
H^{2}(5) = &5*4*3*2*1 \text{\hspace{3mm}} 4*3*2*1 \text{\hspace{3mm}} 3*2*1 \text{\hspace{3mm}} 2*1 \text{\hspace{3mm}}1 \\
&5*4*3*2 \text{\hspace{8mm}} 4*3*2 \text{\hspace{8mm}} 3*2 \text{\hspace{8mm}} 2 \\
&5*4*3 \text{\hspace{13mm}} 4*3 \text{\hspace{13mm}} 3 \\
&5*4 \text{\hspace{18mm}} 4 \\
&5 \\
\end{split}
\end{equation*}
\item \begin{equation*}
\begin{split}
M(8)= &1*2*3*4*5*6*7*8 \\
&1*2*3\hspace{13.5mm}6*7*8\\
&1*2\hspace{23.6mm}7*8\\
&1\hspace{33.7mm}8\\
\end{split}
\end{equation*}
\item \begin{equation*}
\begin{split}
M(9)= &1*2*3*4*5*6*7*8*9 \\
&1*2*3*4\hspace{8.5mm}6*7*8*9\\
&1*2*3\hspace{18.6mm}7*8*9\\
&1*2\hspace{28.7mm}8*9\\
&1\hspace{38.8mm}9
\end{split}
\end{equation*}
\item \begin{equation*}
\begin{split}
N(7)= &1*2*3*4*5*6*7 \\
 &\hspace{5mm} 2*3*4*5*6 \\
&\hspace{10mm}3*4*5 \\
&\hspace{15mm}4
\end{split}    
\end{equation*}
\item \begin{equation*}
\begin{split}
N(6)= &1*2*3*4*5*6 \\
&\hspace{5mm}2*3*4*5 \\
&\hspace{10mm}3*4
\end{split}    
\end{equation*}
\end{enumerate}
\vspace{5mm}
Item numbers 1 to 4 offers a visual insight into the pattern generated by the superfactorial and hyperfactorial (Degrees 1 and 2) functions.
No.5 and No.6 portray a pattern not yet discussed in this paper. It (the pattern) is generated by the function M(x). On observation of values of M(1) to M(10), it will become apparent that an explicit formula certainly exists which is associated with a certain combination of the factorial, floor and ceiling functions. Indeed, the explicit formula happens to be:
\begin{equation*}
\begin{split}
M(x)= x! \left[\prod^{\lceil \frac{x}{2} \rceil-1}_{k=1}\left(\left(\frac{x!}{\left(\lfloor \frac{x}{2} \rfloor+k\right)!}\right)\left(\lceil \frac{x}{2} \rceil-k\right)!\right) \right]
\end{split}
\end{equation*}
Item numbers 7 and 8 is yet another pattern which grows (as in, the pattern extends) in a nigh complimentary fashion to the aforementioned number pattern created by the function M(x). This function's explicit formula happens to be:
\begin{equation*}
\begin{split}
N(x)=\prod^{\lceil \frac{x}{2} \rceil-1}_{k=0}\frac{(x-k)!}{k!}
\end{split}
\end{equation*}
\begin{center}
\line(1,0){360}
\end{center}
\specialsection{\textbf{The Generalized Primorial Function}}
\begin{definition}
The primorial of $x$, $\mathfrak{p}(x)$, is defined as the product of the first x prime numbers [with $P_{k}$ referring to the $k^{th}$ prime]:
\begin{equation*}
    \begin{split}
        \mathfrak{p}(x)=\prod^{x}_{k=1}\left(P_{k}\right)
    \end{split}
\end{equation*}
\end{definition}
\vspace{15mm}

Before the introduction of the generalized primorial function, it is necessary to mention two other functions related to the prime factors of numbers: $\omega(x)$ evaluating the number of distinct prime factors of n, and $\Omega(x)$ evaluating the total number of prime factors of n. Note that the following result (Result 6.1) holds true \textit{because of the fact that all factors of a number's primorial - by definition, the product of n distinct primes - are prime and distinct}.
\begin{equation}
\omega(\mathfrak{p}(x))=\Omega(\mathfrak{p}(x))=x 
\end{equation}
We now introduce the generalized (a.k.a. $n^{th}$-degree) primorial function.

\begin{definition}
Define the $n^{th}$-degree primorial as follows: \begin{equation}
\mathfrak{p}^{(n)}(x)=\prod^{x}_{k=1}\mathfrak{p}^{(n-1)}(P_k) \hspace{2mm}; \hspace{2mm} \mathfrak{p}^{(0)}(P_k)=P_k
\end{equation}
\end{definition}

Note that $\mathfrak{p}^{(1)}(x)$ is, in fact, the original primorial function $\mathfrak{p}(x)$. Now, based on the recently defined function, we introduce two more generalized functions.
\begin{definition}
Define the total number of distinct prime factors of the $n^{th}$ degree primorial of a natural number x as follows:
$$\alpha^{(n)}_{\mathfrak{p}}(x)=\omega(\mathfrak{p}^{(n)}(x))$$
\end{definition}
\begin{definition}
Define the total number of prime divisors of the $n^{th}$ degree primorial of a natural number x as follows:
$$\beta^{(n)}_{\mathfrak{p}}(x)=\Omega(\mathfrak{p}^{(n)}(x))$$
\end{definition}

We now take a look at a theorem involving the $\alpha^{(n)}_{\mathfrak{p}}(x)$ function, preceded by a required lemma and its proof.
\begin{center}
\line(1,0){360}
\end{center}
\begin{lemma}
$\alpha^{(n+1)}_{\mathfrak{p}}(x) = \alpha^{(n)}_{\mathfrak{p}}(P_x)$
\end{lemma}
\begin{proof}
\begin{equation*}
\begin{split}
\alpha^{(n+1)}_{\mathfrak{p}}(x) & = \omega(\mathfrak{p}^{(n+1)}(x))\\
    & = \omega\left(\prod^x_{i=1}\mathfrak{p}^{(n)}(P_i)\right)\\
    & = \omega\left(\left(\mathfrak{p}^{(n)}(P_1)\right)\left(\mathfrak{p}^{(n)}(P_2)\right) \cdot\cdot\cdot \left(\mathfrak{p}^{(n)}(P_{x-1})\right)\left(\mathfrak{p}^{(n)}(P_x)\right)\right)
\end{split}
\end{equation*}
Now, observe that:
\begin{equation*}
    \begin{split}
    \mathfrak{p}^{(n)}(P_x) & = \prod^{P_x}_{i=1}\mathfrak{p}^{(n-1)}(P_i)\\
    & = \left(\mathfrak{p}^{(n-1)}(P_1)\right)\left(\mathfrak{p}^{(n-1)}(P_2)\right) \cdot\cdot\cdot \mathlarger{\left(\mathfrak{p}^{(n-1)}(P_{(P_{x-1}-1)})\right)}\\
    & \mathlarger{\left(\mathfrak{p}^{(n-1)}(P_{(P_{x-1})})\right) \left(\mathfrak{p}^{(n-1)}(P_{(P_{x-1}+1)})\right)} \cdot\cdot\cdot \left(\mathfrak{p}^{(n-1)}(P_{(P_{x}-1)})\right)\left(\mathfrak{p}^{(n-1)}(P_{P_x})\right)\\
    & = \prod^{P_{x-1}}_{i=1}\mathfrak{p}^{(n-1)}(P_i) \left(\mathfrak{p}^{(n-1)}(P_{(P_{x-1}+1)})\right) \cdot\cdot\cdot \left(\mathfrak{p}^{(n-1)}(P_{(P_{x}-1)})\right)\left(\mathfrak{p}^{(n-1)}(P_{P_x})\right)\\
    & = \prod^{P_{x-1}}_{i=1}\mathfrak{p}^{(n-1)}(P_i) \mathlarger{\prod^{P_x}_{i=P_{(x-1)}+1}\mathfrak{p}^{(n-1)}(P_i)}\\
    & = \mathfrak{p}^{(n)}(P_{x-1}) \mathlarger{\prod^{P_x}_{i=P_{(x-1)}+1}\mathfrak{p}^{(n-1)}(P_i)} \hspace{5mm} \text{By definition (6.2)}\\
    \end{split}
\end{equation*}
The above factorization lets us know that $\left(\mathfrak{p}^{(n)}(P_{x-1})\right)$ is a factor of $\left(\mathfrak{p}^{(n)}(P_x)\right)$. On repeated usage of the aforementioned observation, it becomes clear that  all the terms that are multiplied with each other, in the argument of the $\omega(x)$ function, are factors of $\left(\mathfrak{p}^{(n)}(P_x)\right)$. Since each term is a product of exclusively prime numbers (and as mentioned before, a factor of $\left(\mathfrak{p}^{(n)}(P_x)\right)$),$\left(\mathfrak{p}^{(n)}(P_x)\right)$ is bound to encompass all the \textit{distinct} prime factors of the $\omega(x)$ function's argument. Ultimately, since $\omega(x)$ function evaluates only the total number of \textit{distinct} prime factors, $\vspace{3mm}$
\begin{equation}
    \omega\left(\left(\mathfrak{p}^{(n)}(P_1)\right)\left(\mathfrak{p}^{(n)}(P_2)\right) \cdot\cdot\cdot \left(\mathfrak{p}^{(n)}(P_x)\right)\right)=\omega(\mathfrak{p}^{(n)}(P_x))=\alpha^{(n)}_{\mathfrak{p}}(P_x)
\end{equation}
\end{proof}
\begin{center}
\line(1,0){360}
\end{center}
\vspace{2mm}
\begin{theorem}
$$\alpha^{(n)}_{\mathfrak{p}}(x)=\mathlarger{\mathlarger{P_{P_{P_{P_{\cdot_{\cdot_{P_x}}}}}}}}=
\mathlarger{\mathlarger{{P_{\alpha^{(n-1)}_{\mathfrak{p}}(x)}}}} \hspace{2mm};\hspace{2mm} n\geq2$$
\end{theorem}

\begin{proof}
Theorem 6.6. is proven by using the Principle of Mathematical Induction.
\vspace{5mm}

\textbf{Base Case: $n=2$} 
To prove: $$\alpha^{(2)}_{\mathfrak{p}}(x)=\mathlarger{\mathlarger{{P_{\alpha^{(1)}_{\mathfrak{p}}(x)}}}}$$
\begin{proof}
Note that $P_{\alpha^{(1)}_{\mathfrak{p}}(x)}=P_x$ by result obtained in (6.1). 
\begin{equation}
\alpha^{(2)}_{\mathfrak{p}}(x)=\omega\left[\mathfrak{p}^{(2)}(x)\right]=\omega\left[\prod^x_{i=1}\mathfrak{p}^{(1)}(P_i)\right]=\omega\left[\prod^x_{i=1}\prod^{P_i}_{k=1}P_k\right]
\end{equation}
The above equality follows from definitions 6.1., 6.2. and 6.3.  
\begin{equation}
\omega\left[\prod^x_{i=1}\prod^{P_i}_{k=1}P_k\right]=\omega\left[\prod^{P_1}_{k=1}P_k\prod^{P_2}_{k=1}P_k \cdot \cdot \cdot\prod^{P_x}_{k=1}P_k\right]
\end{equation}
Now, note that: $\mathlarger{\prod^{P_x}_{k=1}P_k}=\left(\mathlarger{\prod^{P_{x-1}}_{k=1}P_k}\right)\left(\mathlarger{\prod^{P_x}_{k=P_{x-1}}P_k}\right)$. Because $\omega(x)$ simply counts total number of \textit{distinct} factors, the right hand side of equation (6.5) is simply equal to $\mathlarger{\omega\left(\prod^{P_x}_{k=1}P_k\right)}$. This is the same as $\omega(\mathfrak{p}^{(1)}(P_x))$ (from definition (6.1)), which is simply $P_x$ (from equation (6.1)). Therefore, 
\begin{equation}
\omega\left[\prod^{P_1}_{k=1}P_k\prod^{P_2}_{k=1}P_k \cdot \cdot \cdot\prod^{P_x}_{k=1}P_k\right]=\omega(\prod^{P_x}_{k=1}P_k)=\omega(\mathfrak{p}^{(1)}(P_x))=P_x
\end{equation}
\end{proof}
\textbf{Inductive Hypothesis: }

\begin{equation*}
\alpha^{(n)}_{\mathfrak{p}}(x)=\mathlarger{\mathlarger{{P_{\alpha^{(n-1)}_{\mathfrak{p}}(x)}}}}
\end{equation*}

\textbf{To prove:} \begin{equation*}
\alpha^{(n+1)}_{\mathfrak{p}}(x)=\mathlarger{\mathlarger{{P_{\alpha^{(n)}_{\mathfrak{p}}(x)}}}}
\end{equation*}
Now,
\begin{equation*}
    \begin{split}
    \alpha^{(n+1)}_{\mathfrak{p}}(x) & = \alpha^{(n)}_{\mathfrak{p}}(P_x) \hspace{5mm} [\text{Lemma (6.5)}]\\
    & = \mathlarger{\mathlarger{P_{\alpha^{(n-1)}_{\mathfrak{p}}(P_x)}}} \hspace{5mm} [\text{From Inductive Hypothesis}]\\
    & = \mathlarger{\mathlarger{P_{\alpha^{(n)}_{\mathfrak{p}}(x)}}} \hspace{5mm} [\text{Lemma (6.5)}]
    \end{split}
\end{equation*}
\end{proof}
\begin{center}
\line(1,0){360}
\end{center}
We now introduce a theorem involving the $\beta^{(n)}_{\mathfrak{p}}(x)$ function (total number of prime divisors of the $n^{th}$ degree primorial of a natural number x), and then give the proof for a lemma which is utilized to complete this proof:
\vspace{2mm}
\begin{theorem}
$$\beta^{(n)}_{\mathfrak{p}}(x)=\mathlarger{\sum^x_{k=1}\beta^{(n-1)}_{\mathfrak{p}}(P_k)} \hspace{2mm};\hspace{2mm} n\geq2$$
\end{theorem}
\begin{proof}

\begin{equation*}
    \begin{split}
    \vspace{4mm}
    \beta^{(n)}_{\mathfrak{p}}(x) & = \Omega(\mathfrak{p}^{(n)}(x))\\
    & = \Omega\left(\prod^x_{k=1}\mathfrak{p}^{(n-1)}(P_k)\right)\\
    & = \Omega\left(\left[\mathfrak{p}^{(n-1)}(P_1)\right] \left[\mathfrak{p}^{(n-1)}(P_2)\right] \cdot\cdot\cdot \left[\mathfrak{p}^{(n-1)}(P_x)\right]\right)\\
    & = \Omega(\mathfrak{p}^{(n-1)}(P_1))+\Omega(\mathfrak{p}^{(n-1)}(P_2))+\cdot\cdot\cdot\Omega(\mathfrak{p}^{(n-1)}(P_x)) \hspace{1mm};  \text{[From Lemma (6.8)]}\\
    & = \beta^{(n-1)}_{\mathfrak{p}}(P_1)+\beta^{(n-1)}_{\mathfrak{p}}(P_2)+ \cdot\cdot\cdot \beta^{(n-1)}_{\mathfrak{p}}(P_x ) \hspace{1mm} ; \text{[From Definition (6.4)]}\\
    & = \sum^x_{k=1}\beta^{(n-1)}_{\mathfrak{p}}(P_k)
    \end{split}
\end{equation*}
\end{proof}   
\begin{lemma}
\begin{equation*}
    \boxed{\Omega\left(\prod^n_{i=1} x_i\right)=\sum^n_{i=1}\Omega(x_i);\hspace{1mm} x_i\geq1, n \in \mathbb{Z}^+}
\end{equation*}
\end{lemma}
\begin{proof}
This proof essentially shows that the $\Omega$ function is \textit{completely additive}. Let $$x_i=\prod^{\infty}_{j=1}P_j^{m_{i_j}}\hspace{1mm},i \in \{1,2,\cdot\cdot\cdot n\}$$ be the canonical representation\textsuperscript{\cite{Can}} of the number $x_i$ where a finite number of $m_{i_j}$ are positive integers \& the rest are 0. $P_j$ is the $j$\textsuperscript{th} prime number. Now, since the $\Omega$ function counts the total number of prime factors of a number, note that: $$\Omega(x_i)=\Omega\left(\prod^{\infty}_{j=1}P_j^{m_{i_j}}\right)=\sum^{\infty}_{j=1}m_{i_j}$$
It is to be noted that $(\sum^{\infty}_{j=1}m_{i_j})$ is a convergent series (i.e. results in a finite sum), since only a finite number of terms are positive integers (rest are 0). The RHS of Lemma 6.8 is, therefore, simply:
\begin{equation}
    \begin{split}
    \sum^n_{i=1}\Omega(x_i)& =\sum^n_{i=1}\left(\sum^{\infty}_{j=1}m_{i_j}\right)=\left(\sum^{\infty}_{j=1}m_{1_j}\right)+\left(\sum^{\infty}_{j=1}m_{2_j}\right)+\cdot\cdot\cdot\left(\sum^{\infty}_{j=1}m_{n_j}\right)\\
    & = \sum^{\infty}_{j=1}\left(m_{1_j}+m_{2_j}+\cdot\cdot\cdot m_{n_j}\right)\\ & = \boxed{ \sum^{\infty}_{j=1}\left(\sum^n_{i=1}m_{i_j}\right)}
    \end{split}
\end{equation}
Now, notice that LHS of Lemma 6.8 is simply:
\begin{equation*}
    \begin{split}
    \Omega\left(\prod^n_{i=1}x_i\right)& = \Omega\left(\prod^n_{i=1} \left(\prod^{\infty}_{j=1}P_j^{m_{i_j}}\right)\right)\\
    & = \Omega\left(\prod^{\infty}_{j=1}P_j^{m_{1_j}}\times\prod^{\infty}_{j=1}P_j^{m_{2_j}}\times\cdot\cdot\cdot\prod^{\infty}_{j=1}P_j^{m_{n_j}}\right)\\
    & = \Omega\left(\prod^{\infty}_{j=1}P_j^{m_{1_j}+m_{2_j}+\cdot\cdot\cdot m_{n_j}}\right)\\& = \Omega\left(\prod^{\infty}_{j=1}P_j^{\sum^n_{i=1}m_{i_j}}\right)\\
    & = \boxed{\sum^{\infty}_{j=1}\left(\sum^n_{i=1}m_{i_j}\right)} 
    \end{split}
\end{equation*}
\end{proof}

\specialsection{\textbf{Modular arithmetic with superfactorials}}

This section utilizes $\boxed{a \mod b = c}$ to mean that $c$ is the unique remainder when $a$ is divided by $b$, such that $0\leq c<b$.

\smallskip
We consider the remainder when the product of the first $x$ positive integers is divided by the sum of the first $x$ integers - i.e., consider $x! \text{ mod } T_x$, where $T_x$ is the $x^{\text{th}}$ triangular number. Note that $T_x=P_2(x)$.
We first introduce three lemmas.

\begin{lemma}
\begin{equation*}
d \mod ab = d \mod a + a\left(\left\lfloor\frac{d}{a} \right\rfloor \mod b\right)
\end{equation*}
\end{lemma}
\begin{proof}
Let $d=p+ka;\hspace{2mm} p \in [0,a-1]$, and we have $p = d \mod a$, while $k=\left\lfloor\frac{d}{a}\right\rfloor$. 
Now let $k=q+mb; \hspace{2mm}q \in [0,b-1]$, and we have $q = k \mod b = \left\lfloor\frac{d}{a}\right\rfloor \mod b$. Now,
\begin{equation*}
    \begin{split}
        d & = p+ka = p+(q+mb)a\\
        & = p + aq + m(ab)\\
    \end{split}
\end{equation*}
But $p+aq \leq (a-1)+a(b-1)=ab-1$.
\begin{equation*}
    \therefore d \mod ab  = p+aq = d \mod a + a\left(\left\lfloor\frac{d}{a} \right\rfloor \mod b\right)
\end{equation*}
\end{proof}
\begin{lemma}
\begin{equation*}
    ax \mod n = ay \implies x \mod \frac{n}{gcd(a,n)} = y
\end{equation*}
\end{lemma}
\begin{proof}
By definition, $ax \mod n = ay \implies ax=ay+kn \implies x = y +\frac{kn}{a}$. 
Now, if $a$ and $n$ are relatively prime, i.e. $gcd(a,n)=1$, $a$ necessarily divides $k$, since:
\begin{equation*}
    ax=ay+kn \implies k=\frac{a(x-y)}{n}
\end{equation*}
We therefore have $\frac{k}{a}\in \mathbb{Z}^+$, and so, with :
\begin{equation*}
    x = y +\frac{kn}{a} \implies x = y +\left(\frac{k}{a}\right)n \implies x \mod n = \boxed{x \mod \frac{n}{gcd(a,n)}= y}
\end{equation*}
If $a$ and $n$ aren't relatively prime, then let $gcd(a,n)=d; \hspace{2mm}a=md$.
Then,$$x = y +\frac{kn}{a} \implies x = y +\frac{kn}{md}$$. Because $m\nmid n$, $m$ must divide $k$. Let $\frac{k}{m}=\alpha$, then we have:
\begin{equation*}
    x = y +\alpha\frac{n}{d} \implies x \mod \frac{n}{d} = \boxed{x \mod \frac{n}{gcd(a,n)}=y} 
\end{equation*}
\end{proof}
\begin{lemma}
For all prime numbers $p$,
\begin {equation*} (p-2)! \mod p = 1 \end{equation*}
\end{lemma}
\begin{proof}
Let $S_p$ be the symmetric group on $p$ elements. The order of $S_p$ then, is $|S_p|=p!=p(p-1)!$. From Sylow's third theorem, we know that the number of Sylow $p$-subgroups $n_p$ of a finite group $G$ divides $m$, where $|G|=p^nm; \hspace{2mm} n>0, p\nmid m$. Therefore for $G=S_p$, it follows that $n_p|(p-1)!$. Note that $(p-1)!$ is the total elements number of order $p$, and since each of the Sylow $p$-subgroups have exactly $p-1$ of these elements, we have:
\begin{equation}
    n_p(p-1)=(p-1)! \implies n_p = (p-2)!\label{eq:1}
\end{equation}
Sylow's third theorem also tells us that $n_p \equiv 1 \pmod p$, and this along with equation $(\ref{eq:1})$ ultimately gives us:
\begin{equation*}
    (p-2)! \equiv 1 \pmod p \implies (p-2)! \mod p = 1
\end{equation*}
\end{proof}
\begin{theorem}
\begin{equation*}
    sf^{(0)}(x)\mod P_2(x) = \boxed{x!\mod P_2(x) = x \left\lfloor \frac{x}{\left(p_{\pi(x)+1}-1\right)} \right\rfloor};\hspace{3mm} x\geq 2, x\in \mathbb{Z}^+
\end{equation*}
\end{theorem}
\begin{proof}
$$\frac{x!}{P_2(x)}=\frac{x!}{
 \frac{x(x+1)}{2}}=\frac{2(x-1)!}{(x+1)}$$
 
Now, consider the following cases:

\begin{enumerate}
    \item $(x+1)$ is an even number.
    $$(x+1)=2\frac{(x+1)}{2}=2k;\hspace{3mm} 1< \frac{(x+1)}{2}\leq(x-1),\hspace{1mm} \forall x\in\mathbb{Z}^+\geq3$$
    This means that $k|(x-1)!, \forall x\in\mathbb{Z}^+\geq3$. Note that $x=1$ gives  $\frac{1!}{\frac{1*(1+1)}{2}}=1$.
    
    \begin{equation}\therefore (x+1)=2k \implies x! \mod P_2(x) = 0\label{eq:2}\end{equation}
    \item $(x+1)$ is an odd composite number. 
    $$(x+1)=gq; \hspace{3mm} 1<\{g,q\}<(x-1), \forall x\in\mathbb{Z}^+>3$$
    \smallskip
    If $g\neq q$, then $qg|(x-1)!$. If $q=g$, note that $(x+1)=g^2 \implies (x-1)=g^2-2$. Note that since $g^2-2>2g$, $g(2g)|(x-1)!$ and therefore $g^2|(x-1)! \implies gq|(x-1)!$.
    
    \begin{equation} 
    \therefore (x+1)=gq \implies x! \mod P_2(x) = 0 \label{eq:3}
    \end{equation}
    
    \item $(x+1)$ is a prime number.

    Notice that in this case, $(x+1)\nmid (x-1)!$, because no number from $1$ to $(x-1)$, as well as their product $(x-1)!$, has $(x+1)$ as their prime factor. Using lemma $(7.1)$, we have: 
    \begin{equation}
    \begin{split}
        x! \mod \frac{x(x+1)}{2} & = x! \mod \frac{x}{2}+\frac{x}{2}\left(\left\lfloor\frac{x!}{\frac{x}{2}}\right\rfloor \mod (x+1)\right)\\
        & = 0 + \frac{x}{2}\left(2(x-1)! \mod (x+1)\right)\\
        & = \frac{x}{2}\left(2l\right)\label{eq:4}
    \end{split}    
    \end{equation}
    Where $\boxed{2(x-1)! \mod (x+1) = 2l}$. Realize that employing lemma $(7.2)$, we get $\boxed{(x-1)! \mod (x+1) = l}$. Further, lemma $(7.3)$ tells us that $l=1$. So from equation $(\ref{eq:4})$, we finally have:
    \begin{equation}
        x! \mod \frac{x(x+1)}{2} = \frac{x}{2}(2l) = x\label{eq:5}
    \end{equation}
\end{enumerate}
We therefore have, from equations $(\ref{eq:2})$,$(\ref{eq:3})$ and $(\ref{eq:5})$:
\begin{equation}
x! \mod \frac{x(x+1)}{2} = \left\{
        \begin{array}{ll}
            0 & \quad x+1 \text{ is composite}  \\
            x & \quad x+1 \text{ is prime}
        \end{array}
    \right. \label{eq:6}
\end{equation}
This tells us that the LHS of $(7.5)$ evaluates to $x$ only if $x+1$ is prime. 
\smallskip

Now, the largest prime number less than or equal to $x$ is simply $p_{\pi(x)}$. So the smallest prime number greater than $x$ is $p_{\pi(x)+1}$. So we can say, that the LHS of equation $(\ref{eq:6})$ evaluates to $x$, only if $x+1=p_{\pi(x)+1} \implies x=p_{\pi(x)+1}-1$. 
\smallskip

Observe that the maximum value attainable by $f(x)=\frac{x}{p_{\pi(x)+1}-1}$ is 1, at $x=p_{\pi(x)+1}-1$. For all other values of x, $f(x)\in(0,1)$. Therefore, taking the floor function of $f(x)$ helps us get different discrete values, and the following holds:
\begin{equation}
    \left\lfloor\frac{x}{p_{\pi(x)+1}-1} \right\rfloor  = \left\{
        \begin{array}{ll}
            0 & \quad x+1 \neq p_{\pi(x)+1} \\
            1 & \quad x+1=p_{\pi(x)+1}
        \end{array}
    \right.
\end{equation}
Multiplying the above equation by $x$, and realizing that this process results in its RHS being the same as that of equation $(\ref{eq:6})$ , we finally get:
$$\boxed{x!\mod P_2(x) = x \left\lfloor \frac{x}{\left(p_{\pi(x)+1}-1\right)} \right\rfloor};\hspace{3mm} x\geq 2, x\in \mathbb{Z}^+$$
\end{proof}
\begin{theorem}
$\forall x>n+1, x\in\mathbb{Z}^+ $,
\begin{equation*}
    \boxed{sf^{(n)}(x) \mod P_{n+2}(x) = 0 \iff \sum^{n+1}_{i=1}\left\lfloor \frac{x}{\left(p_{\pi(x)+1}-i\right)} \right\rfloor \left\lfloor \frac{\left(p_{\pi(x)+1}-i\right)}{x} \right\rfloor=0}
\end{equation*}
\end{theorem}
\begin{proof}
First, let the falling factorial $x\fallingfactorial{n}$ be defined as follows: $$x\fallingfactorial{n}={\overbrace{x(x-1)\dots(x-(n-1))}^{\text{$n$ factors}}}$$
Note that: $$P_{n+2}(x)=\binom{x+n+1}{n+2}=\frac{(x+n+1)!}{(n+2)!(x-1)!}=\frac{(x+n+1)(x+n)\cdot\cdot\cdot(x+1)x}{(n+2)!}$$
$$\therefore P_{n+2}(x)= (x+n+1)\fallingfactorial{n+1}\left[\frac{x}{(n+2)!}\right]$$
Now,
\begin{equation}
    g(x)=\frac{sf^{(n)}(x)}{P_{n+2}(x)}=\ddfrac{(n+2)!\prod^{x-1}_{k=1}k^{P_n((x-k)+1)}}{(x+n+1)\fallingfactorial{n+1}}
\end{equation}
It is apparent in the above equation that if and only if any of the terms in the denominator  $(x+n+1)\fallingfactorial{n+1}$ is a prime number, then $g(x)$ is not an integer. 
\begin{align}
    \therefore \left(g(x)\not\in\mathbb{Z}^+ \iff sf^{(n)}(x) \mod P_{n+2}(x) \neq 0\right)\iff\\ \left(g(x)\in\mathbb{Z}^+ \iff sf^{(n)}(x) \mod P_{n+2}(x) = 0\right) \label{eq:7}
\end{align}
Now consider the terms $\left\lfloor \frac{x}{\left(p_{\pi(x)+1}-i\right)} \right\rfloor$ and $\left\lfloor \frac{\left(p_{\pi(x)+1}-i\right)}{x} \right\rfloor$, and note that $\left\lfloor\ddfrac{a}{b}\right\rfloor\left\lfloor\ddfrac{b}{a}\right\rfloor \neq 0 \iff a=b$. Realize that $\pi(x)+1$ is the smallest prime greater than $x$, and that for an integer $i\in[1,n+1]$, $x=p_{\pi(x)+1}-i$ would imply that one of the terms in $(x+n+1)\fallingfactorial{n+1}$ is a prime (with the prime being $x+i$). So for none of the terms in $(x+n+1)\fallingfactorial{n+1}$ to be a prime, it would require that:

\begin{equation}
    \left\lfloor \frac{x}{\left(p_{\pi(x)+1}-i\right)} \right\rfloor\left\lfloor \frac{\left(p_{\pi(x)+1}-i\right)}{x} \right\rfloor=0, \hspace{1mm} \forall i\in\mathbb{Z}^+;1\leq i\leq n+1 \label{eq:8}
\end{equation}
So we ultimately have, $\forall x>n+1, x\in\mathbb{Z}^+ $, based on $(\ref{eq:7})$ and $(\ref{eq:8})$,

\begin{equation*}\boxed{
    sf^{(n)}(x) \mod P_{n+2}(x) = 0 \iff \sum^{n+1}_{i=1}\left\lfloor \frac{x}{\left(p_{\pi(x)+1}-i\right)} \right\rfloor \left\lfloor \frac{\left(p_{\pi(x)+1}-i\right)}{x} \right\rfloor=0}
\end{equation*}
\end{proof}

\specialsection{\textbf{On the factors of superfactorials}}

\textit{Legendre's formula} (otherwise referred to as \textit{de Polignac's formula}), gives us the number of times a prime number $p$ divides into $x!$:

\begin{equation}
    \nu_p(x!)=\sum^{\infty}_{i=1}\left\lfloor\frac{x}{p^i}\right\rfloor \label{eq:9}
\end{equation}
Where $\nu_p(x)$ is the standard notion for the \textit{p-adic} order of $x$.
Let 
\begin{equation*}
    ^{sf}\nu_p^{(n)}(x) \equiv \nu_p\left(sf^{(n)}(x)\right)
\end{equation*}
Then using this notation, we see that \textit{Legendre's formula} involves $^{sf}\nu_p^{(0)}(x)$, since $sf^{(0)}(x)=x!$. We now present a proof of a sum-expression for $^{sf}\nu_p^{(n)}(x)$.

\begin{theorem}
\begin{equation*}
    ^{sf}\nu_p^{(n)}(x) = \sum^x_{i=1} \left[ P_{n-1}(x+1-i)\left(\sum^{\infty}_{k=1}\left\lfloor\frac{i}{p^k}\right\rfloor\right)\right]; \hspace{2mm} n \in \mathbb{Z}^+
\end{equation*}
\end{theorem}
\begin{proof}
\textbf{Base Case: $n=1$} \begin{equation*}
    ^{sf}\nu_p^{(1)}(x) = \nu_p\left(sf(x)\right) = \sum_{i=1}^x \nu_p(i!) = \sum_{i=1}^x P_0(x+1-i)\left(\sum^{\infty}_{k=1}\left\lfloor\frac{i}{p^k}\right\rfloor\right)
\end{equation*}
The above follows from $\ref{eq:9}$, and the fact that $\nu_p(a \cdot b)=\nu_p(a)+\nu_p(b)$ (where $a$ \& $b$ are integers).

\textbf{Inductive Hypothesis:}
\begin{equation*}
    ^{sf}\nu_p^{(m)}(x) = \sum^x_{i=1} \left[ P_{m-1}(x+1-i)\left(\sum^{\infty}_{k=1}\left\lfloor\frac{i}{p^k}\right\rfloor\right)\right]; \hspace{2mm} m \in \mathbb{Z}^+
\end{equation*}
\textbf{To Prove:}
\begin{equation*}
    ^{sf}\nu_p^{(m+1)}(x) = \sum^x_{i=1} \left[ P_{m}(x+1-i)\left(\sum^{\infty}_{k=1}\left\lfloor\frac{i}{p^k}\right\rfloor\right)\right]; \hspace{2mm} (m+1) \in \mathbb{Z}^+
\end{equation*}
Now, 
\begin{equation*}
\begin{split}
    ^{sf}\nu_p^{(m+1)}(x) & = \sum^x_{i=1}\nu_p\left(sf^{(m)}(i)\right) = \sum^x_{i=1}{}^{sf}\nu_p^{(m)}(i)\\
    & = \sum^x_{i=1}\sum^i_{j=1} \left[ P_{m-1}(i+1-j)\left(\sum^{\infty}_{k=1}\left\lfloor\frac{j}{p^k}\right\rfloor\right)\right]\\
    & = \sum^x_{j=1}\left[\left(\sum^{\infty}_{k=1}\left\lfloor\frac{j}{p^k}\right\rfloor\right)\sum^x_{i=j}  P_{m-1}(i+1-j)\right]\\
    \end{split}
\end{equation*} 
However, note that:
\begin{equation*}
\begin{split}
    \sum^x_{i=j}  P_{m-1}(i+1-j) & = P_{m-1}(j+1-j) + P_{m-1}((j+1)+1-j) \cdot \cdot \cdot \cdot\\ & \hspace{4mm}+ P_{m-1}((x-1)+1-j) + P_{m-1}(x+1-j)\\
    & = P_{m-1}(1) + P_{m-1}(2) \cdot \cdot \cdot + P_{m-1}(x+1-j) = \sum^{x+1-j}_{i=1}  P_{m-1}(i)\\
    & = P_{m}(x+1-j) \hspace{2mm} [\text{From equation \ref{eq:main}}]
\end{split}
\end{equation*}
Therefore, as required, we have:
\begin{equation*}
    ^{sf}\nu_p^{(m+1)}(x) = \sum^x_{j=1}\left[\left(\sum^{\infty}_{k=1}\left\lfloor\frac{j}{p^k}\right\rfloor\right)P_{m}(x+1-j)\right]
\end{equation*}
\end{proof}
\begin{corollary}
\begin{equation*}
    ^{sf}\nu_p^{(n)}(x) = \sum^x_{i=1} \left[ P_{n-j}(x+1-i)\nu_p^{(j-1)}(x)\right]; \hspace{2mm} \{n,j\} \in \mathbb{Z}^+, n\geq j
\end{equation*}
\end{corollary}
The above is merely the result of exercise in algebra. 
\bigskip

Another formula to express $\nu_p(x!)$ would be the following \textsuperscript{\cite{Mihet}}: $\nu_p(x!)=\frac{x-s_p(x)}{p-1}$, where $s_p(x)$ is the sum of digits of the base-$p$ expansion of $x$, i.e.
 $$s_p(x)=\sum^{\left\lfloor \log_p(x) \right\rfloor}_{n=0}\frac{1}{b^n}(x \mod b^{n+1}-x \mod b^n)$$
 Extending this, we get the following corollary centred around $^{sf}\nu_p^{(n)}(x)$, which can be proven in a fashion similar to that of Theorem $8.1$.
\begin{corollary}
\begin{equation*}\boxed{
    ^{sf}\nu_p^{(n)}(x) = \frac{1}{p-1}\left[P_{n+1}(x)-\sum^x_{a_n=1}\sum^{a_n}_{a_{n-1}=1}\cdot\cdot\cdot\sum^{a_2}_{a_1=1}s_p(a_1) \right]; \hspace{2mm} n \in \mathbb{Z}^+}
\end{equation*}
\end{corollary}


\begin{thebibliography}{12}
\bibitem{supfac} 
Weisstein, Eric W. "Superfactorial." From MathWorld, A Wolfram Web Resource. 
\texttt{http://mathworld.wolfram.com/Superfactorial.html}
\bibitem{sloane1} 
N.J.A. Sloane,
Sequence A000178 in The On-Line Encyclopedia of Integer Sequences,
\texttt{http://oeis.org/A000178}
\bibitem{hypfac} 
Weisstein, Eric W. "Hyperfactorial." From MathWorld, A Wolfram Web Resource.
\texttt{http://mathworld.wolfram.com/Hyperfactorial.html}
\bibitem{sloane2} 
N.J.A. Sloane,
Sequence A002109 in The On-Line Encyclopedia of Integer Sequences, Alan Sokal
\texttt{http://oeis.org/A002109}
\bibitem{sloane3} 
N.J.A. Sloane,
Sequence A002109 in The On-Line Encyclopedia of Integer Sequences, Vaclav Kotesovec
\texttt{http://oeis.org/A002109}
\bibitem{fignum} 
Weisstein, Eric W. "Figurate Number." From MathWorld--A Wolfram Web Resource.
\texttt{http://mathworld.wolfram.com/FigurateNumber.html}
\bibitem{sloane4}
Sloane, N. J. A. and Plouffe, S. Figure M2535 in The Encyclopedia of Integer Sequences. San Diego: Academic Press, 1995.
\bibitem{Can}
Apostol, Tom M., Introduction to Analytic Number Theory - Springer International Student Edition, Eighth Narosa Publishing House Reprint (1998), ISBN: 978-81-85015-12-5
\bibitem{Mihet}
Miheţ, Dorel. "Legendre’s and Kummer’s theorems again." Resonance 15.12 (2010): 1111-1121.
\end{thebibliography}
\end{document}